\documentclass[12pt]{amsart}
\usepackage{amsmath, amssymb}
\numberwithin{equation}{section}

\theoremstyle{plain}
\newtheorem*{t-theorem}{Theorem}
\newtheorem{theorem}{Theorem}[section]
\newtheorem{proposition}[theorem]{Proposition}
\newtheorem{corollary}[theorem]{Corollary}
\newtheorem{lemma}[theorem]{Lemma}

\theoremstyle{definition}

\newtheorem{remark}[theorem]{Remark}

\def \N {\mathbb{N}}
\def \R {\mathbb{R}}

\def \E {\mathbb{E}}
\def \P {\mathbb{P} \, }

\def \NN {\mathcal{N}}
\def \MM {\mathcal{M}}

\def \a {\alpha}

\def \e {\varepsilon}
\def \d {\delta}
\def \D {\Delta}

\def \l {\lambda}

\def \t {\tau}

\def \Om {\Omega}

\def \etc {,\ldots,}

\def \unconv {{\rm unc.conv}}

\def \conv {{\rm conv}}
\def \absconv {{\rm abs.conv}}

\def \vol {{\rm vol}}

\newcommand{\pr}[2]{\langle {#1} , {#2} \rangle}
\newcommand{\norm}[1]{\left \| #1 \right \|}
\title[Complexity of the set of unconditional convex bodies]{On the complexity of the set of unconditional convex bodies
   }

\author{Mark Rudelson}\thanks{ Department of Mathematics, University of Michigan. Partially supported by NSF grants DMS-01161372, DMS-1464514,  and
USAF Grant FA9550-14-1-0009.}

\begin{document}

\begin{abstract}
 We show that for any $1 \le t \le \tilde{c} n^{1/2} \log^{-5/2} n$, the set of unconditional convex bodies in $\R^n$ contains a $t$-separated subset of cardinality at least
 \[
   \exp \left( \exp \left( \frac{c}{t^2 \log^4 (1+t)} \, n \right) \right).
 \]
 This implies the existence of an unconditional convex body in $\R^n$ which cannot be approximated within the distance $d$ by a projection of a polytope with $N$ faces unless $N \ge \exp(c(d) n)$.

 We also show that for $t \ge 2$, the cardinality of a $t$-separated set of completely symmetric bodies in $\R^n$ does not exceed
   \[
    \exp \left( \exp \left( C \frac{\log^2 n}{\log t}  \right) \right).
  \]
\end{abstract}

\maketitle

\section{introduction}

In \cite{BV} Barvinok and Veomett posed a question whether any $n$-dimen\-sional convex symmetric body can be approximated by a projection of a section of  a simplex whose dimension is subexponential in $n$. The importance of this question stems from the fact that the convex bodies generated this way allow an efficient construction of the membership oracle. The question of Barvinok and Veomett has been answered in \cite{LRT}, where it was shown
that for all $1\leq n\leq N$,
there exists an $n$-dimensional symmetric convex body $B$ such that
for every $n$-dimensional convex body $K$ obtained as a projection of a
section of an $N$-dimensional simplex one has
$$
    d(B, K) \geq  c    \sqrt{  \frac{ n }{ \ln \frac{2 N \ln (2N)}{n}  }} ,
$$
where $d(\cdot, \cdot)$ denotes the Banach-Mazur distance
and $c$ is an absolute positive constant. Moreover, this result is sharp up to
a logarithmic factor.

One of the main steps in the proof of this result was an estimate of the complexity of the set of all convex symmetric bodies in $\R^n$, i. e., the Minkowski or Banach--Mazur compactum.
The complexity is measured in terms of the maximal size of a $t$-separated set with respect to the Banach--Mazur distance
\[
 d(K,D)= \inf \{\l \ge 1 \mid  D \subset TK \subset \l D \},
\]
where the infimum is taken over all linear operators $T: \R^n \to \R^n$.
A set $A$ in a metric space $(X,d)$ is called $t$-separated if the distance between any two distinct points of $A$ is at least $t$.
It follows from  \cite{LRT} that for any $1 \le t \le cn$, the set of all $n$-dimensional convex bodies contains a $t$-separated subset of cardinality at least
\begin{equation}
  \label{eq: entropy-Minkowski}
  \exp(\exp(cn/t)).
\end{equation}
More precisely, Theorem 2.3 \cite{LRT} asserts that for any $2n \le M \le e^n$, there exists a probability measure $\P_M$ on the set of convex symmetric polytopes such that
\begin{equation}
 \label{eq: Gluskin distance}
 \P_M \otimes \P_M \left( \left\{(K', K'') \mid  d(K',
     K'')\leq  \frac{c \, n}{\ln (M/n)}
\right\}\right) \leq   2e^{- n M} .
\end{equation}
  This probabilistic estimate together with the union bound implies the required lower bound on the maximal size of a $t$-net.

Note that for $t=O(1)$, the estimate above shows that the complexity of the Minkowski compactum is doubly exponential in terms of the dimension. This fact has been independently established by Pisier \cite{P}, who asked whether a similar statement holds for the set of all unconditional convex bodies and for the set of all completely symmetric bodies. We show below that the answer to the first question is affirmative, and to the second one negative.

Consider unconditional convex bodies first.
A convex symmetric body $K \subset \R^n$ is called unconditional if it symmetric with respect to all coordinate hyperplanes. This property can be conveniently reformulated in terms of the norm generated by $K$. For $x \in \R^n$, set
\[
 \norm{x}_K = \min \{a \ge 0 \mid x \in aK \}.
\]
The body $K$ is unconditional if for any $x=\sum_{j=1}^n x_j e_j$, and for any $J \subset [n]$,
\[
  \norm{\sum_{j \in J} x_j e_j}_K \le   \norm{ \sum_{j \in [n]} x_j e_j}_K,
\]
i.e., whenever all coordinate projections are contractions.

Our main result shows that the complexity of the set $\mathcal{K}_n^{unc}$ of unconditional convex bodies at the scale $t$ is doubly exponential as long as $t=O(1)$. More precisely, we prove the following theorem.
\begin{theorem}
 \label{thm: entropy}
 Let $1 \le t \le \tilde{c} n^{1/2} \log^{-5/2} n$.
 The set of $n$-dimensional unconditional convex bodies contains a $t$-separated set of cardinality at least
 \[
   \exp \left( \exp \left( \frac{c}{t^2 \log^4 (1+t)} \, n \right) \right).
 \]
 Here, $\tilde{c}$ and $c$ are positive absolute constants.
\end{theorem}

Note that unlike the estimate \eqref{eq: entropy-Minkowski}, which is valid for $1 \le t \le cn$, the estimate above holds only in the range $1 \le t \le \tilde{c} n^{1/2} \log^{-5/2} n$. By a theorem of Lindenstrauss and Szankowski \cite{LS}, the maximal Banach--Mazur distance between two $n$-dimensional unconditional bodies does not exceed $Cn^{1-\e_0}$ for some  $\e_0 \ge 1/3$. This means that a non-trivial estimate of the cardinality of a $t$-separated set in $\mathcal{K}_n^{unc}$ is impossible whenever $t> n^{1-\e_0}$.

 Following the derivation of Theorem 1.1 \cite{BV}, one can show that Theorem \ref{thm: entropy} implies a result on the hardness of approximation of an unconditional convex body by a projection of a section of a simplex refining the solution of the problem posed by Barvinok and Veomett.

 \begin{corollary}
 \label{cor: Barvinok}
Let $n\leq N$. There exists an $n$-dimensional  unconditional convex  body $B$, such that
for every $n$-dimensional convex body $K$ obtained as a projection of a
section of an $N$-dimensional simplex one has
\[
    d(B, K) \geq  c
     \left( \frac{n }{  \log N} \right)^{1/4} \cdot \log^{-1} \left( \frac{n}{\log N} \right)  ,
\]
where $c$ is an absolute positive constant.
\end{corollary}

In particular, Corollary \ref{cor: Barvinok} means that to be able to approximate all unconditional convex bodies in $\R^n$ by projections of sections of an $N$-dimensional simplex within the distance $O(1)$, one has to take $N \ge \exp(cn)$.

Consider now the set of completely symmetric bodies.
We will call an $n$-dimensional convex body  completely symmetric if it is unconditional and invariant under all permutations of the coordinates. This term is not commonly used. In the language of normed spaces, completely symmetric convex bodies correspond to the spaces with 1-symmetric basis. However, since the term ``symmetric convex bodies'' has a different meaning, we will use ``completely symmetric'' for this class of bodies.

 The set of completely symmetric convex bodies is much smaller than the set of all unconditional ones.
  This manifests quantitatively in the fact that the cardinality of a $t$-separated set of completely symmetric bodies is significantly lower. Namely, we prove the following proposition in Section \ref{seq: completely symmetric}.
 \begin{proposition}
  \label{prop: comp-symm}
  Let $t \ge 2$. The cardinality of any $t$-separated set in $\mathcal{K}^{cs}$ does not exceed
  \[
    \exp \left( \exp \left( C \frac{\log^2 n}{\log t} \right) \right).
  \]
 \end{proposition}
This proposition means, in particular, that the complexity of the set of completely symmetric convex bodies is not doubly exponential in the dimension, which answers the second question of Pisier.

\subsection*{Acknowledgement} The author is grateful to Olivier Gu\'edon for several suggestions which allowed to clarify the presentation.

\section{Notation and an outline of the construction}

 Let us list some basic notation  used in the proofs below. By $\P$ and $\E$ we denote the probability and the expectation. If $N$ is a natural number, then $[N]$ denotes the set of all integers from 1 to $N$.

Let $1 \le p \le \infty$. By $\norm{x}_p$ we denote the standard $\ell_p$-norm of a vector $x=(x_1 \etc x_n) \in \R^n$:
\[
 \norm{x}_p= \left( \sum_{j=1}^n |x_j|^p \right)^{1/p},
\]
and $B_p^n$ denotes the unit ball of $\ell_p^n$.
 If $A: \R^n \to \R^m$ is a linear operator, and $K_1, K_2$ are convex symmetric bodies, then $\norm{A: K_1 \to K_2}$ stands for the operator norm of $A$ considered as an operator between normed spaces with unit balls $K_1$ and $K_2$:
\[
 \norm{A: K_1 \to K_2}= \max_{x \in K_1} \norm{Ax}_{K_2}.
\]
The norm $\norm{A: B_2^n \to B_2^m}$ is denoted simply by $\norm{A}$. The Hilbert--Schmidt norm of $A$ is
\[
  \norm{A}_{HS}= \left( \sum_{i=1}^m \sum_{j=1}^n |a_{ij}|^2 \right)^{1/2}.
\]
For  $K_1 \etc K_L \subset \R^n$, we denote by $\text{abs.conv}(K_1 \etc K_L)$ their absolute convex hull
\[
  \text{abs.conv}(K_1 \etc K_L) = \{ \sum_{l=1}^L \l_l x_l \mid
  \sum_{l=1}^L |\l_l| \le 1,\  x_l \in K_l \ \text{for } l \in [L] \}.
\]
Also, we define the unconditional convex  hull of the points $x^1 \etc x^L \in \R^n$ with coordinates $x^j=(x^j_1 \etc x^j_n)$ by
\[
  \text{unc.conv}(x^1 \etc x^L)=
  \text{conv} \left((\e^1_1 x^1_1 \etc \e^1_n x^1_n) \etc (\e^L_1 x^L_1 \etc \e^L_n x^L_n) \right),
\]
where the convex hull is taken over all choices of $\e_{i}^l \in \{-1,1\}$. Obviously, $\text{unc.conv}(x^1 \etc x_L)$ is the smallest unconditional convex set containing $x^1 \etc x^L$.

Finally, $C,c, c_0$ etc. denote absolute constants whose value may change from line to line.

 The random convex bodies $K',K''$ appearing in \eqref{eq: Gluskin distance} are generalized Gluskin polytopes. Such polytopes were introduced by Gluskin \cite{Gl} to prove that the diameter of Minkowski compactum is of the order $\Om(n)$. These polytopes are constructed as the absolute convex hull of $N(n)$ independent vectors uniformly distributed over $S^{n-1}$ and a few deterministic unit vectors. Such construction, however, cannot be adopted to prove Theorem \ref{thm: entropy}. Indeed, an argument based on measure concentration shows that if  $x_1 \etc x_N, x_1' \etc x_N'$ are independent random vectors uniformly distributed over $S^{n-1}$ and $N \ge n$, then with high probability
 \[
  d( \text{unc.conv}(x_1 \etc x_N),  \text{unc.conv}(x_1' \etc x_N')) \le C,
 \]
 making it impossible to achieve distances greater than $O(1)$. Moreover, if $N/n \to \infty$, then this distance tends to 1, which does not allow to prove the doubly exponential complexity bound for distances of order $O(1)$ either.

 To avoid the problems arising in attempts to use the standard construction of the Gluskin polytopes, we give up on the assumption that the random vectors are uniformly distributed on the sphere. Instead, we fix  number $\d>0$ and $N \in \N$ depending on the desired distance and consider independent random sets $I_1 \etc I_N \subset [n]$ uniformly chosen among the sets of cardinality $\d n$.
 Here and below, we assume for simplicity that the numbers $\d n, n/2$ etc. are integer.
 Alternatively, one can take the integer part of these numbers.
 For each $l \in [N]$, set $x_l=\sum_{i \in I_l} e_i$. The random convex body $K$ will be defined as
 \begin{multline*}
  K= K(I_1 \etc I_N) \\
  = \absconv \left( \unconv (x_1 \etc x_N), \sqrt{\d n} B_1^n, \d \sqrt{n} B_2^n \right),
 \end{multline*}
 Here the scaled copies of $B_1^n$ and $B_2^n$ appear only for technical reasons, and the main role is played by the unconditional convex hull of $x_1 \etc x_N$. The main advantage of this construction is that the distance between two independent copies of such bodies can be large and can be controlled in terms of $\d$ and $N$.

  One of the important features of this construction is that the random points $x_l, \ l \in [N]$ are defined via random sets $I_l$ of a fixed cardinality.
  This means that the coordinates of $x_l$ are not independent. An alternative definition of random vertices $y_l=\sum_{i=1}^n \nu_{i,l} e_i$, where $\nu_{i,l}, \  i \in [n], l \in [N]$ are independent Bernoulli random variables taking value $1$ with probability $\d$ would have been much easier to work with.
  Yet, with such definition, $\P(\nu_{1,1}= \cdots = \nu_{n,1}=1)=\d^n$, which is only exponentially small in $n$.
  This would have made the doubly exponential bound for probability unattainable.

 We will show in Section \ref{sec: random polytopes} that the distance between two independent copies of the polytope $K$ is large with probability close to 1. This will allow us to derive Theorem \ref{thm: entropy} by an application of the union bound. The large deviation and small ball probability estimates instrumental for the proof of the main result of Section \ref{sec: random polytopes} are obtained in Section \ref{sec: deviation}.

\section{Small ball probability and large deviation bounds \\ for the linear image of a random vector}
\label{sec: deviation}

We start with establishing a concentration estimate for random quadratic forms similar to the Hanson--Wright inequality.
The following Lemma is based on Theorem 1.1 \cite{RV}.

\begin{lemma}   \label{l: random set}
 Let $J$ be a random subset of $[n]$ of size $m<n$ uniformly chosen among all such subsets.
   Denote  by $R_J=\sum_{j \in j} e_j e_j^T$ the coordinate projection on the set $J$.
  Let $Y=(\e_1 \etc \e_n)$ be the vector whose coordinates are independent symmetric $\pm 1$ Bernoulli random variables.
   Then for any $n \times n$ matrix $A$ and any $t>0$,
 \begin{align*}
   &\P \left(  \left| Y^T R_J A R_J Y - \E  Y^T R_J A R_J Y \right| \ge t \right) \\
   &    \le 2 \exp \left[ -c \left( \frac{t^2}{m \norm{A}^2} \bigwedge \frac{t}{\norm{A}} \right) \right].
 \end{align*}
\end{lemma}

\begin{proof}
 Let us separate the diagonal and off-diagonal terms. We have
  \begin{align*}
   &\P \left(  \left| Y^T R_J A R_J Y - \E  Y^T R_J A R_J Y \right| \ge t \right) \\
   &\le \P \left( \left| \sum_{j \in J} a_{jj} - \frac{m}{n}  \sum_{j=1}^n a_{jj} \right| \ge \frac{t}{2} \right)
   + \P \left( \left| \sum_{j \in J, j \neq k} \e_j \e_k a_{jj}  \right| \ge \frac{t}{2} \right) \\
   &=: p_1+p_2.
  \end{align*}
  We estimate $p_1$ and $p_2$ separately. To estimate $p_1$, consider a function $F$ on the permutation group $\Pi_n$ defined by
  \[
   F(\pi)=\sum_{j=1}^m a_{\pi(j),\pi(j)}.
  \]
  For $k<n$, denote by $\mathcal{A}_l$ the algebra of subsets of $\Pi_n$, whose elements are the sets of permutations $\pi$ for which $\pi(1) \etc \pi(l)$ is the same. Let $X_0(\pi)=\E F(\pi)$, and for $l \le m$, set $X_l=\E[F(\pi) \mid \mathcal{A}_l]$. The sequence $X_0 \etc X_m$ defined this way is a martingale with martingale differences
  \[
   |X_{l+1}-X_l| \le \max_{j \in [n]} |a_{jj}| \le \norm{A}.
  \]
  Hence, by Azuma's inequality
  \[
   p_1=\P(|X_m-X_0| \ge t/2)
   \le 2 \exp \left(- \frac{c t^2}{m \norm{A}^2} \right).
  \]
  The estimate for $p_2$ follows directly from Theorem 1.1 \cite{RV}. Indeed, let $Y_J$ be the coordinate restriction of the vector $Y$ to the set $J$. Similarly, let $A_J$ be the square submatrix of the matrix $A$ whose rows and columns belong to $J$. Denote by $\D_J$ the diagonal of $A_J$. Then
  \[
   \norm{A_J-\D_J} \le 2 \norm{A} \quad \text{and} \quad
   \norm{A_J-\D_J}_{HS} \le \sqrt{|J|} \cdot \norm{A_J-\D_J} \le 2 \sqrt{m} \norm{A}.
  \]
  Therefore, by Theorem 1.1 \cite{RV} we have
  \begin{align*}
    p_2
    &=\P \left(|Y_J^T(A_J-\D_J)Y_J| \ge t/2 \right)
    =\E_J \P \left[|Y_J^T(A_J-\D_J)Y_J| \ge t/2 \mid J \right] \\
    &\le 2 \exp \left[ -c \left( \frac{t^2}{m \norm{A}^2} \bigwedge \frac{t}{\norm{A}} \right) \right].
  \end{align*}
\end{proof}
The small ball probability bound follows immediately from Lemma \ref{l: random set}.
 \begin{corollary} \label{cor: small ball}
 Let $B$ be an $n \times n$ matrix.
  In the notation of Lemma \ref{l: random set},
  \[
   \P \left( \norm{B R_J Y}_2 \le \sqrt{\frac{m}{2n}} \norm{B}_{HS} \right)
   \le 2 \exp \left( -c \frac{m}{n^2} \cdot \frac{\norm{B}_{HS}^4}{\norm{B}^4}  \right).
  \]
 \end{corollary}
 \begin{proof}
  We apply Lemma \ref{l: random set} with $A=B^T B$, and $t=\frac{m}{2n} \norm{B}_{HS}^2$. In this case, $\E  Y^T R_J A R_J Y= \frac{m}{n} \text{tr}(A) = 2t$.
  Then the left hand side of the inequality above can be bounded by
  \[
       \exp \left[ -c \left( \frac{m}{n^2} \cdot \frac{\norm{B}_{HS}^4}{\norm{B}^4}
    \bigwedge
         \frac{ m}{n}  \cdot \frac{\norm{B}_{HS}^2}{\norm{B}^2} \right) \right].
  \]
  To derive the corollary, note that $\norm{B}_{HS}^2 \le n \norm{B}^2$, so the first term in the minimum is always smaller than the second one.
 \end{proof}
 Lemma \ref{l: random set} can be also applied to derive the large deviation inequality for $\norm{B R_J Y}_2$. However, the bound obtained this way will not be strong enough for our purposes. To prove the large deviation estimate we employ a different technique.

\begin{lemma}   \label{l: large deviation}
 Let $J$ be a random subset of $[n]$ of size $m<n$ uniformly chosen among all such subsets.
   Denote  by $R_J=\sum_{j \in j} e_j e_j^T$ the coordinate projection on the set $J$.
  Let $Y=(\e_1 \etc \e_n)$ be vector  whose coordinates are independent symmetric $\pm 1$ Bernoulli random variables.
   Then for any $n \times n$  matrix $B$ and any $t> \sqrt{4m/n}  \cdot \norm{B}_{HS}$,
 \[
   \P \left(  \norm{BR_J Y}_2 \ge t \right) \\
       \le 2 \exp \left( -\frac{c t^2}{\norm{B}^2} \right).
 \]
\end{lemma}

\begin{proof}
 Condition on the set $J$ first. Note that $\norm{B R_J} \le \norm{B}$. Applying Talagrand's convex distance inequality \cite{Tal}, we obtain
 \[
   \P( |\norm{BY}_2- M| \ge s \mid J) \le 2 \exp \left(-\frac{s^2}{2\norm{B}^2} \right),
 \]
 where $M$ is the median of $\norm{B R_J Y}_2$. Since $M \le (\E \norm{B R_J Y}_2^2)^{1/2}=\norm{BR_J}_{HS}$, the previous inequality implies
 \begin{equation}  \label{i: cond on J}
   \P \left(    \norm{BR_J Y}_2 \ge t + \norm{B R_J}_{HS} \mid J \right)
   \le 2 \exp \left(-\frac{ct^2}{\norm{B}^2} \right).
 \end{equation}
 Set $A=B^T B$.
 To finish the proof, we have to obtain a large deviation bound for the random variable
 \[
   U:= \norm{B R_J}_{HS}^2= \text{tr}(R_J^T A R_J)= \sum_{j \in J} a_{jj}
 \]
 depending on $J$.
 The set $J$ is chosen uniformly from the sets of cardinality $m$, so the elements of $J$ are not independent.

 To take advantage of independence, let us introduce auxiliary random variables. Let $\d_1 \etc \d_n$ be independent Bernoulli $\{0,1\}$ random variables taking  value 1 with probability $2m/n$. Then $\P(\sum_{j=1}^n \d_j <m)<1/2$. Chernoff's inequality provides more precise bound for this probability, but this estimate would suffice for our purposes.  Set
 \[
   Z=F(\d_1 \etc \d_n) =\sum_{j=1}^n \d_j \ a_{jj}.
 \]
 Since $\max_{j \in [n]} a_{jj} \le \norm{A}$, we derive from Bernstein's inequality that
 \[
   \P ( Z>\t) \le \exp \left( \frac{-c \t}{\norm{A}} \right)
 \]
    for any $\t \ge  \frac{4m}{n} \cdot \text{tr}(A) = 2 \E Z$.

 Compare the random variables $U$ and $Z$. Notice that the random variable $Z$ conditioned on the event $\sum_{j=1}^m \d_j=m$ has the same distribution as $U$. Also, for any $m'<m''$,
 \[
  \P(Z>\t \mid \sum_{j=1}^m \d_j=m') \le \P(Z>\t \mid \sum_{j=1}^m \d_j=m'').
 \]
 This observation allows to conclude that for any $\t \ge   \frac{4m}{n} \cdot \text{tr}(A)$,
 \begin{align*}
    &\P \left(Z>\t \mid \sum_{j=1}^m \d_j=m \right) \cdot \sum_{m' \ge m} \P( \d_j=m') \\
    \le
    &\sum_{m' \ge m} \P \left(Z>\t \mid \sum_{j=1}^m \d_j=m'\right) \cdot \P( \d_j=m')
     \le \exp \left( \frac{-c \t}{\norm{A}} \right).
 \end{align*}
 Thus,
 \begin{align*}
  \P(U > \t)
  &=\P \left(Z>\t \mid \sum_{j=1}^m \d_j=m \right)  \\
  &\le \left( 1-\P(\sum_{j=1}^m \d_j<m)\right)^{-1} \exp \left( \frac{-c \t}{\norm{A}} \right)
  \le 2 \exp \left( \frac{-c \t}{\norm{A}} \right).
 \end{align*}
 Combining this with inequality \eqref{i: cond on J}, we obtain that for any $t> \sqrt{4m/n}  \cdot \norm{B}_{HS}$,
 \begin{align*}
    &\P \left(  \norm{BR_J Y}_2 \ge 2t \right) \\
    &\le E_J \P \left(   \norm{BR_J Y}_2 \ge t + \norm{B R_J}_{HS} \mid \norm{B R_J}_{HS} \le t \right)
    +\P (\norm{B R_J}_{HS}^2 > t^2)
    \\
    &\le 3 \exp \left( \frac{-c t^2}{\norm{B}^2} \right).
 \end{align*}

\end{proof}

\section{Distance between unconditional random polytopes} \label{sec: random polytopes}

We follow the classical scheme of estimating the distances developed for Gluskin's  polytopes, see e.g.  \cite{MTJ}.
Fix an $n \times n$ matrix $V \in GL(n)$.
Denote the singular values of $V$ by  $\norm{V}=s_1(V) \ge s_2(V) \ge \cdots \ge s_n(V)>0$.
For the Banach--Mazur distance estimate, we can normalize $V$ by assuming $s_{n/2}(V) \ge 1$ and $s_{n/2}(V^{-1}) \ge 1$.
Let $K$ and $K'$ be independent random unconditional convex bodies, and let $Y$ be a vertex of $K$.
We start with estimating the probability that $V Y \subset d K'$ for some $d>1$.
For the standard Gluskin polytopes, such estimate is obtained by volumetric considerations.
In our setting, this argument is unavailable, and we use the results of Section \ref{sec: deviation} instead.

\begin{proposition} \label{prop: one vector}
 Let $\d > C n^{-1/2}$, and let $N= \exp(c \d^2 n)$. For $l \in [N]$, let $I_l \subset[n]$ be a set of cardinality $|I_l|=m=\d n$. Define a convex body $\tilde{K}$ by
 \[
  \tilde{K}= \tilde{K}(I_1 \etc I_l)
   =\absconv \left( \sqrt{\d n} B_2^{I_1} \etc  \sqrt{\d n} B_2^{I_N}, \,  \d \sqrt{n}  B_2^n \right).
 \]
 Let $J \subset [n]$ be a random subset of cardinality $\d n$ uniformly distributed in $[n]$, and let $\e_1 \etc \e_n$ be independent symmetric $\pm 1$ Bernoulli random variables. Set
 \[
  Y= \sum_{j \in J} \e_j e_j.
 \]
 Then for any linear operator $V: \R^n \to \R^n$ with $s_{n/2}(V) \ge 1$ and $\log \norm{V} \le 1/\sqrt{\d}$,
 \[
  \P \left(VY \in \frac{c}{\sqrt{\d} \log \norm{V}} \tilde{K} \right) \le \exp(-c' \d^2 n).
 \]
\end{proposition}

\begin{proof}
 Let
 \[
  V= \sum_{i=1}^n s_j u_i v_i^{\intercal}
 \]
 be the singular value decomposition of $V$. Then there exists an interval $I =[i_1,i_2] \subset [1,n/2]$
  of cardinality
 \[
  i_2-i_1=|I|=: r\ge \frac{c_0 n}{ \log \norm{V}}
 \]
 such that $s_{i_1}/s_{i_2} \le 2$. Indeed, otherwise we would have
 \[
  \norm{V} \ge \prod_{k=1}^{n/(2r)} \frac{s_{(k-1)r+1}}{s_{kr}} \ge 2^{n/(2r)} > \norm{V}
 \]
 provided that the constant $c_0$ is chosen small enough. Set
 \[
  Q= \sum_{i \in I} s_i u_i v_i^{\intercal} \quad \text{and } P= \sum_{i \in I}  u_i u_i^{\intercal}
 \]
Since the ratio of the maximal and the minimal singular values of $Q$ does not exceed 2,  the operator $Q$ satisfies
 \[
  r \norm{Q}^2 \ge \norm{Q}_{HS}^2 \ge (r/4) \norm{Q}^2.
 \]
 Note that for any $\a>0$, $VY \in \a \tilde{K}$ implies
 \[
  QY= PQY = PVY \in \a P \tilde{K}.
\]
 Corollary \ref{cor: small ball} applied to $B=Q$ yields
 \[
 \P (\norm{Q Y}_2 \le s_{i_2} \cdot c \sqrt{\d r} )
 \le \exp \left( -c \frac{m}{n^2} r^2 \right)
 \le (-\tilde{c} \d^2 n),
 \]
 where the last  inequality follows from the definition of $r$ and the assumption $\log \norm{V} \le 1/\sqrt{\d}$.

 For $l \in [N]$, set $E_l= \text{span} \{P e_j: \ j \in I_l \}$, and let $P_{E_l}$ be the orthogonal projection onto $E_l$.
Since $\norm{Q} =s_{i_1} \le 2 s_{i_2}$,
 \[
   \norm{P_{E_l} Q} \le 2 s_{i_2} \quad \text{and} \quad
    \norm{P_{E_l}Q}_{HS} \le 2s_{i_2} \cdot \norm{P_{E_l}}_{HS}
  \le 2 s_{i_2} \cdot \sqrt{\d n}.
 \]
  Applying Lemma \ref{l: large deviation} with $t=C s_{i_2} \d \sqrt{n} >  2 \sqrt{\d} \norm{P_{E_l} Q}_{HS}$, we get
 \[
  \P( \norm{P_{E_l}Q Y}_2 \ge C s_{i_2}  \d \cdot  \sqrt{ n}) \le e^{-\tilde{c} \d^2 n}.
 \]
 Let $\Om$ be the event that
 \begin{enumerate}
  \item $\norm{Q Y}_2 \ge s_{i_2} \cdot c \sqrt{\d r}$ and
  \item $\norm{P_{E_l}Q Y}_2 \le C s_{i_2} \cdot \d \sqrt{ n}$ for any $l \in [N]$.
 \end{enumerate}
 The previous estimates show that
 \begin{equation}
   \label{eq: prob complement Om}
  \P(\Om^c) \le  e^{-\tilde{c} \d^2 n} + N e^{-\tilde{c} \d^2 n} \le e^{-c' \d^2 n},
 \end{equation}
 where we used the assumption on $N$ with a sufficiently small constant~$c$.

 Since $P B_2^{I_l} \subset B_2^n \cap E_l$, for any $y \in \R^n$,  we have
 \[
   \max_{x \in B_2^{I_l}} \pr{Qy}{x} = \max_{x \in B_2^{I_l}} \pr{Qy}{Px}
   \le \max_{u \in B_2^n \cap E_l} \pr{Qy}{u} = \norm{P_{E_l} Q y}_2.
 \]
 Assume that $\Om$ occurs.
 Conditions (1) and (2) imply
  \begin{equation}
    \label{eq: PQ}
    \norm{P_{E_l}QY}_2 \le c \sqrt{n/r} \cdot \sqrt{\d} \norm{QY}_2
  \end{equation}
  for all $l \in [N]$.
 If $QY \in \a P \tilde{K}$, then
 \begin{align*}
     \norm{Q Y}_2^2
     &\le \a \max_{x \in P \tilde{K}} |\pr{QY}{x}|
     =  \a \max_{x \in  \tilde{K}} |\pr{QY}{x}|   \\
    &\le \a \left( \max_{l \in [N]} \max_{x \in \sqrt{\d n} B_2^{I_l}}|\pr{Q Y}{x}|
       + \max _{x \in \d \sqrt{n} B_2^n}|\pr{Q Y}{x}| \right) \\
    &\le \a \sqrt{\d n} \left( \max_{l \in [N]} \norm{P_{E_l} Q Y}_2+ \sqrt{\d} \norm{Q Y}_2 \right)\\
    &\le C \sqrt{\frac{n}{r}} \cdot \a \d \sqrt{n} \norm{Q Y}_2,
 \end{align*}
 where the last inequality holds because of \eqref{eq: PQ}.
  Combining this with (1) and recalling that $s_{i_2} \ge 1$, we obtain
 \[
   \a \ge \tilde{c} \frac{r}{n} \cdot \frac{1}{\sqrt{\d}}
   > \frac{c}{\sqrt{\d} \log \norm{V}}
 \]
 if $c$ is chosen small enough.
 This means that the event $VY \in \frac{c}{\sqrt{\d} \log \norm{V}} \tilde{K}$ implies $\Om^c$, and so, the proposition follows from estimate \eqref{eq: prob complement Om}.
\end{proof}

Proposition \ref{prop: one vector} pertains to one random vector $Y$.
The body $K= K(I_1 \etc I_N)$ contains many independent copies of $Y$, and we can use this independence to derive the desired doubly exponential bound for probability.
If $K$ and $K'$ are independent convex bodies, denote by $\P_K$ the probability with respect to $K$ conditioned on $K'$ being fixed.
\begin{corollary}  \label{cor: one body}
  Let $\d > C n^{-1/2}$, and let $N= \exp(c \d^2 n)$.
 Let $I_l, \ l \in [N]$ be independent random subsets of $[n]$ uniformly chosen among the sets of cardinality $\d n$.
 For $l \in [N]$ denote $x_l=\sum_{j \in I_l} e_j$. Consider a random convex body
 \begin{multline*}
  K= K(I_1 \etc I_N)
  = \conv \left( \unconv (x_1 \etc x_N), \sqrt{\d n} B_1^n, \d \sqrt{n} B_2^n \right),
 \end{multline*}
 and let $K'$ be an independent copy of $K$.
 Assume that $K'=K(I_1' \etc I_N')$ satisfies $\bigcup_{l=1}^N I_l' =[n]$.

 Let $V : \R^n \to \R^n$ be a linear operator such that $s_{n/2}(V) \ge 1$.  Then
 \[
  \P_K \left( V K \subset \frac{c_0}{\sqrt{\d} \log (1/\d)} K' \right)
  \le \exp \Big( - \exp( c \d^2 n) \Big).
 \]
\end{corollary}

\begin{proof}
Denote for shortness
\[
 \a= \frac{c}{\sqrt{\d} \log (1/\d)}.
\]
 Assume that $VK \subset \a K'$. Since $K \supset \d \sqrt{n} B_2^n$ and $K' \subset \sqrt{\d n} B_2^n$, we have
 $
  \norm{V} \le \a/\sqrt{\d}
 $,
 which implies
 \[
  \log \norm{V} \le C \log (1/\d).
 \]
 The condition on $K'$ implies that $K' \subset \tilde{K}(I_1 \etc I_N)$,
 where the last set is defined in Proposition \ref{prop: one vector}.

 Let $\e_{j,l}, \ j \in [n], \ l \in [N]$ be independent symmetric $\pm 1$  random variables.
 Then the vectors $Y_l = \sum_{j \in I_l} \e_j e_j$ are contained in $K$. Hence, by Proposition \ref{prop: one vector},
 \begin{align*}
    \P(V K \subset \a K')
    &\le \prod_{l=1}^N \P(V Y_l \in \a K') \le \Big( \exp(-c\d^2 n) \Big)^N \\
    &\le \exp \Big( - \exp(c \d^2 n) \Big)
 \end{align*}
 as required.

\end{proof}

 Our main technical result, Theorem \ref{thm: lower bound} below, will imply Theorem~\ref{thm: entropy} almost immediately.
\begin{theorem} \label{thm: lower bound}
 Let $\d > C \sqrt{\frac{\log n}{n}}$, and let $N= \exp(c \d^2 n)$.
 For $l \in [N]$, let $I_l \subset[n]$ be a set of cardinality $|I_l|=\d n$.
 For $l \in [N]$ denote $x_l=\sum_{j \in I_l} e_j$. Consider a random convex body
 \begin{multline*}
  K= K(I_1 \etc I_N) \\
  = \absconv \left( \unconv (x_1 \etc x_N), \sqrt{\d n} B_1^n, \d \sqrt{n} B_2^n \right),
 \end{multline*}
 and let $K'$ be an independent copy of $K$.
 Then
 \[
  \P_{K,K'} \left(d(K,K') \le  \frac{c_1}{\d \log^2 (1/\d)}  \right)
  \le \exp \Big( - \exp( c_2 \d^2 n) \Big).
 \]
\end{theorem}
\begin{proof}
 The proof follows the general scheme developed by Gluskin.
 Fix $K'=K(I_1' \etc I_n')$ such that $\cup_{l=1}^n I_l'=[n]$.
 Let $c$ be a constant to be chosen later.
 Denote by $W(K)$ the event
 \[
  W(K, K')= \{ \exists V: \R^n \to \R^n \ s_{n/2}(V) \ge 1 \text{ and } VK \subset  \frac{c}{\sqrt{\d} \log (1/\d)} K' \}.
 \]
 We start with proving the following \\
 \textbf{Claim.} $\P_K(W(K, K')) \le \exp \Big( - \exp( c \d^2 n) \Big)$. \\

 Set $\e=\frac{c_0}{2 \log (1/\d)}$, where $c_0$ is the constant from Corollary \ref{cor: one body}. By Corollary 8 \cite{MTJ}, the set of all operators $V': \R^n \to \R^n$ such that $ s_{n/2}(V) \ge \sqrt{\d}/\e$ and $V'(\sqrt{\d n} B_1^n) \subset K'$ possesses an $\sqrt{\d}$-net $\NN'$  of cardinality
 \[
  |\NN'| \le \left( \frac{c}{\sqrt{\d}} \right)^{n^2} \cdot \left( \frac{\vol(K'/\sqrt{\d n} )}{ \vol(B_2^n)} \right)^n
  \le C^{n^2},
 \]
 where we used $K' \subset \d n B_1^n$ to obtain the last inequality.
 Hence, $W(K,K')$ possesses an $\e$-net $\NN=(\e/\sqrt{\d}) \NN'$  of cardinality $C^{n^2}$.

 Assume now that there exists an operator $V: \R^n \to \R^n$ such that $s_{n/2}(V) \ge 1$ and $VK \subset (\e/\sqrt{\d}) K'$.
 Let $V_0 \in \NN$ be such that $\norm{V-V_0} <\e$. Then
 \begin{align*}
  \norm{V_0: K \to K'}
  &\le \norm{V: K \to K'} + \norm{V-V_0: K \to K'}  \\
  &\le \frac{\e}{\sqrt{\d}} + \norm{V-V_0: \sqrt{\d n} B_2^n \to \d \sqrt{n} B_2^n} \\
  &\le \frac{2\e}{\sqrt{\d}} =   \frac{c_0}{\sqrt{\d} \log (1/\d)}.
 \end{align*}
 This means that
 \begin{align*}
   \P_K (W(K, K'))
   &\le \P_K ( \exists V_0 \in \NN \ V_0K \subset  \frac{c_0}{\sqrt{\d} \log (1/\d)} K') \\
   &\le |\NN| \cdot \max_{V_0 \in \NN} \P_K (  V_0K \subset  \frac{c_0}{\sqrt{\d} \log (1/\d)} K')
 \end{align*}
 Combining the bound for $|\NN|$ appearing above with Corollary \ref{cor: one body}, we show that this probability does not exceed $\exp \Big( - \exp( c \d^2 n) \Big)$ provided that $\d > C \sqrt{\frac{\log n}{n}}$ for a sufficiently large $C$.
 This completes the proof of the Claim.

 To derive the Theorem from the Claim, note that the inequality $d(K,K') \le  \frac{c}{\d \log^2 (1/\d)}$ guarantees the existence of a linear operator $V: \R^n \to \R^n$ such that
 \[
   \norm{V: K \to K'} \cdot \norm{V^{-1}: K' \to K} \le  \frac{c}{\d \log^2 (1/\d)}.
 \]
 Without loss of generality, we may assume that $s_{n/2}(V) \ge 1$ and $s_{n/2}(V^{-1}) \ge 1$. This means that the event $W(K,K')$ occurs with operator $V$, or $ W(K',K)$ occurs with $V^{-1}$.

Also,
 \[
   \P_{K'} \left(\bigcup_{l=1}^N I_l' \neq [n] \right) \le n (1-\d)^N \le \exp \Big( - \exp( c \d^2 n) \Big).
 \]
 Therefore,
 \begin{align*}
   &\P_{K,K'}  \left(d(K,K') \le  \frac{c}{\d \log^2 (1/\d)}  \right) \\
   &\le \E_{K'} \P_{K} \left( W(K,K') \mid \bigcup_{l=1}^N I_l' = [n] \right) + \P_{K'} \left(\bigcup_{l=1}^N I_l' \neq [n] \right) \\
   &\ + \E_{K} \P_{K'} \left( W(K',K) \mid \bigcup_{l=1}^N I_l = [n] \right) + \P_{K} \left(\bigcup_{l=1}^N I_l \neq [n] \right) \\
   &\le 4 \exp \Big( - \exp( c \d^2 n) \Big).
 \end{align*}
 Theorem \ref{thm: lower bound} is proved.

\end{proof}

\begin{proof}[Proof of Theorem \ref{thm: entropy}]
Let 
\[
   \d = \frac{c_1}{t \log^2 (1+t)},
\]
where $c_1$ is the constant from Theorem \ref{thm: lower bound}.
Choosing the constant $\tilde{c}$ in the formulation of Theorem \ref{thm: entropy} sufficiently small, we ensure that the condition  $\d > C \sqrt{\frac{\log n}{n}}$ holds, and Theorem \ref{thm: lower bound} applies. 
Set 
\[
  M
  = \exp \left( \exp  \left( \frac{c_2}{4} \d^2 n \right) \right)
 =  \exp \left( \exp \left( \frac{c}{t^2 \log^4 (1+ t)} \, n \right) \right),
\]
with some constant $c>0$.

 Consider $M$ independent  unconditional random convex bodies $K_1 \etc K_M$ which are constructed as in Theorem \ref{thm: lower bound}. By this theorem and the union bound,
 \begin{align*}
    & \P \left(\exists m, \bar{m} \in [M] \ m \ne \bar{m}, \ d(K_m,K_{\bar{m}}) \le  \frac{c}{\d \log^2 (1/\d)}  \right)
  \\
  &\le M^2 \cdot \exp \Big( - \exp( c_2 \d^2 n) \Big)
  \le \exp \left( - \exp \left( \frac{c_2}{2} \d^2 n \right) \right).
 \end{align*}
 This inequality implies that $\mathcal{K}_n^{unc}$ contains a $t$-separated set of cardinality at least $M$.
\end{proof}

We now pass to the proof of Corollary \ref{cor: Barvinok}.
Since this corollary follows from Theorem  \ref{thm: entropy} and \cite{LRT}, we will provide only a sketch of the proof instead of a complete argument.
\begin{proof}[Proof of Corollary \ref{cor: Barvinok} (sketch)]
 Fix $m, \ n \le m \le N$.
 Following the proof of Theorem 1.1 \cite{LRT}, we  estimate of the cardinality of a special $2$-net in the set of all $n$-dimensional projections of $m$-dimensional sections of the simplex $\D_N \subset \R^N$.
 By Lemmas 3.1, 3.2 \cite{LRT}, the sections of the $N$-dimensional simplex can be encoded by the points of the Grassmanian $G_{N+1,m}$ so that an $\e$-net $\mathcal{A}_1$ on the Grassmanian corresponds to a $(1+\e m \sqrt{N+1})^2$-net $\NN_1$ in the set of the sections of the simplex in the Banach-Mazur distance.
 Similarly, by Lemma 3.3 \cite{LRT}, for any $K \in \NN_1$ we can encode all $n$-dimensional projections of $K$ by points of the Grassmanin $G_{m,n}$ so that an $\e$-net $\mathcal{A}_2$ on this Grassmanian corresponds to $(1+\e m \sqrt{N+1})^2$-net in the set of the projections of $K$ in the Banach-Mazur distance.
 Combining these two results, we see that the points of $\mathcal{A}_1 \times \mathcal{A}_2$ correspond to some $(1+\e m \sqrt{N+1})^4$-net in the set of the projections of of the sections of the simplex. The nets $\mathcal{A}_1, \mathcal{A}_2$ can be chosen so that
 \[
  |\mathcal{A}_1 \times \mathcal{A}_2| \le \left( \frac{C}{\e} \right)^{\tilde{C} [(N+1)m +m n]} \le \exp \left( C' N^2 \log \frac{C}{\e} \right).
 \]
 Choosing $\e$ such that $(1+\e m \sqrt{N+1})^4=2$, we derive that there exists a $2$-net $\MM_m$ in the set of all $n$-dimensional projections of $m$-dimensional sections of $\D_N$ of cardinality
 \[
  |\MM_m| \le \exp (cN^2 \log N).
 \]
 Setting $\MM= \cup_{m=n}^N \MM_m$, we obtain a $2$-net in the set of all $n$-dimensional projections of sections of $\D_N$ satisfying a similar estimate.

 If any $n$-dimensional unconditional convex body can be approximated by a projection of a section of $\D_N$ within the distance $d$, then $\MM$ is a $(2d)$-net in the set $\mathcal{K}^{unc}_n$. This means that the cardinality of any $(2d)^2$-separated set in $\mathcal{K}^{unc}_n$ does not exceed $ \exp (cN^2 \log N)$.
 
 Let us show that this implies the desired lower bound on $d$.
 Assume that
   \[
    d \le  c'
     \left( \frac{n }{  \log N} \right)^{1/4} \cdot \log^{-1} \left( \frac{n}{\log N} \right),
\]
where the constant $c'>0$ will be chosen later.
If $c'$ is  sufficiently small, then the inequality
 \[
   (2d)^2 \le \tilde{c} n^{1/2} \log^{-5/2} n
 \]
 holds, and so Theorem \ref{thm: entropy} applies.
 By this theorem, there exists a $(2d)^2$-separated set of cardinality at least
 \[
   \exp \left( \exp \left( \frac{cn}{(2d)^4 \log^4((2d)^2+1)} \right) \right).
 \]
 Combining this with the upper estimate for this cardinality proved above, we obtain
 \[
  \frac{cn}{(2d)^4 \log^4((2d)^2+1)} \le 3 \log N.
 \]
 This contradicts our assumption on $d$ if the constant $c'$ is chosen sufficiently small.
\end{proof}

\section{Complexity of the set of completely symmetric bodies}
\label{seq: completely symmetric}

 In this section, we prove Proposition \ref{prop: comp-symm} establishing the upper bound on the cardinality of a $t$-separated set in the set of all completely symmetric bodies.
 Denote the set of all completely symmetric bodies in $\R^n$ by $\mathcal{K}^{cs}$.

\begin{proof}
Fix $1< \t<n$.
 Let $L \in \N$ be the smallest number such that
 \begin{equation}
  \label{eq: def L}
   n \t^{-L} < 1-\t^{-1}.
 \end{equation}
 Denote by $Y$ the set of all non-decreasing functions $\psi: [L] \to [n]$ such that $\psi(1)<\psi(n)$. Note that
 \[
  |Y| \le \binom{n+L}{L} \le (n+L)^L.
 \]
 Define a function $\Phi: \mathcal{K}^{cs} \to \R^Y$ by
 \[
  \Phi_{\psi}(K)=\norm{\sum_{l=1}^L \t^{-l} \sum_{\psi(l-1)<j\le \psi(l)} e_j}_K, \quad \psi \in Y, \ K \in \mathcal{K}^{cs}
 \]
 where we use the convention $\psi(0)=0$.

 Assume that $K,D \in \mathcal{K}^{cs}$.
 We will prove that if
 \[
   \t^{-1} \Phi_{\psi}(D) \le \Phi_{\psi}(K) \le \t \Phi_{\psi}(D)
\]
 for all $\psi \in Y$, then
 \begin{equation}  \label{eq: dist-cs}
  d(K,D) \le \t^6.
 \end{equation}
 To this end, take any vector $x \in \R^n$ such that $\norm{x}_K=1$ and $x_1 \ge \cdots \ge x_n \ge0$.
 Define the vector $y=\sum_{j=1}^n y_j$ with coordinates $y_j$ taking values in the set $\{0\} \cup \{\t^{-l}, \ l \in [L] \}$ so that
 \begin{itemize}
   \item $y_j \le x_j <\t y_j$, if $|x_j| \ge \t^{-L}$;
   \item $y_j=0$, if $x_j <\t^{-L}$.
 \end{itemize}
 Then
 \[
  \norm{x}_K \le \t \norm{y}_K+ \norm{\sum_{x_j < \t^{-L}} x_j e_j}_K
  \le \t \norm{y}_K+  n \t^{-L}
    \le \t \norm{y}_K+ 1- \t^{-1},
 \]
 where we used \eqref{eq: def L} in the last inequality.
 Hence, $1= \norm{x}_K \le \t^2 \norm{y}_K$.
 Also, the inequalities $\Phi_{\psi}(K) \le \t \Phi_{\psi}(D), \ \psi \in Y$ imply $\norm{y}_K \le \t \norm{y}_D$.
 Combining this with $\norm{y}_D \le \norm{x}_D$, we obtain $\norm{x}_K \le \t^3 \norm{x}_D$, and reversing the roles of $K$ and $D$, we derive
 \[
  \t^{-3} \norm{x}_D \le \norm{x}_K \le \t^3 \norm{x}_D,
 \]
 which implies \eqref{eq: dist-cs}.

 Define now a new function $\Theta: \mathcal{K}^{cs} \to \R^Y$ by  setting $\Theta_\psi(K) =\log \Phi_\psi(K)$.
 Then $\Theta(\mathcal{K}^{cs}) \subset [- \log (\t^2), \log n]^Y$. Hence, there exists a $(\log \t)$-net $\NN \subset \Theta(\mathcal{K}^{cs})$ in the $\norm{\cdot}_{\infty}$-norm of cardinality
 \[
  |\NN| \le \left( \frac{\log n}{\log \t} +2 \right)^{|Y|}.
 \]
 For any $x \in \NN$, choose a body $K_x \in \mathcal{K}^{cs}$ such that $\Theta(K_x)=x$ and consider the set $\MM=\{K_x: \ x \in \NN\}$.
 Then for any $K \in \mathcal{K}^{cs}$, there exists $K_x \in \MM$ with $\norm{\Theta(K)-\Theta(K_x)}_{\infty} \le \log \t$.
 This means that for any $\psi \in Y$,  $\t^{-1} \Phi_{\psi}(K_x) \le \Phi_{\psi}(K) \le \t \Phi_{\psi}(K_x)$ for all $\psi \in Y$, and by \eqref{eq: dist-cs}, $d(K,K_x) \le \t^6$.
 Thus, we constructed a $\t^6$-net $\MM$ in $\mathcal{K}^{cs}$ of cardinality
 \[
  |\MM| \le \left( \frac{\log n}{\log \t} +2 \right)^{|Y|}
  \le \left( \frac{\log n}{\log \t} +2 \right)^{(n+L)^L}.
 \]
 Assume now that $\t \ge 2^{1/12}$. Then, by \eqref{eq: def L}, $L<c \frac{\log n}{\log \t}$, and the previous inequality implies
 \[
  |\MM| \le \exp \left( \exp \left( C \frac{\log^2 n}{\log \t} \right) \right).
 \]
 By the multiplicative triangle inequality, the same inequality holds for the cardinality of any $\t^{12}$-separated set in $\mathcal{K}^{cs}$.
 To derive the statement of the Proposition, set $\t=t^{1/12}$.
\end{proof}
\begin{remark}
 The same proof works for all values $t>1$. However, in the case $1<t \le 2$ the estimate of $L$ in \eqref{eq: def L} in terms of $n$ and $t$ is different, which leads to a different estimate of the cardinality of a $t$-separated set.
\end{remark}

\end{document}